\documentclass[letterpaper,11pt]{amsart}

\usepackage{graphicx} 

\usepackage[margin=1.2in]{geometry}
\usepackage{amsmath,amsthm,amssymb}
\usepackage{xspace,xcolor}
\usepackage[breaklinks,colorlinks,citecolor=teal,linkcolor=teal,urlcolor=teal,pagebackref,hyperindex]{hyperref}
\usepackage[alphabetic]{amsrefs}
\usepackage{xypic}
\usepackage{tikz-cd}
\usepackage{comment}

\usepackage{color}

\theoremstyle{plain}
\newtheorem{thm}{Theorem}[section]

\newtheorem{lem}[thm]{Lemma}
\newtheorem{prop}[thm]{Proposition}

\theoremstyle{definition}

\newtheorem{question}[thm]{Question}

\theoremstyle{remark}
\newtheorem{rmk}[thm]{Remark}

\newtheorem{claim}[thm]{Claim}

\def\A{{\mathbb A}}

\def\k{{\mathbb K}}

\def\cO{\mathcal{O}}

\def\.{\cdot}
\def\^{\widehat}

\def\({\left(}
\def\){\right)}

\renewcommand{\and}{ \ \ \text{ and } \ \ }

\usepackage{mathrsfs}

\begin{document}

\author[D.~Bath]{Daniel Bath}
\address{KU Leuven, Departement Wiskunde, Celestijnenlaan 200B, Leuven
3001, Belgium}
\email{dan.bath@kuleuven.be}

\author[M.~Musta\c{t}\u{a}]{Mircea Musta\c{t}\u{a}}
\address{Department of Mathematics, University of Michigan, 530 Church Street, Ann Arbor, MI 48109, USA}
\email{mmustata@umich.edu}

\author[U.~Walther]{Uli Walther}
\address{Purdue University, Department of Mathematics, 150 N. University St.,
West Lafayette, IN 47907, USA}
\email{walther@purdue.edu}

\date{December 14, 2024}

\title{Singularities of square-free polynomials}

\thanks{D.B.\ was supported by FWO grant \#12E9623N. M.M.\ was partially supported by NSF grant DMS-2301463 and by the Simons Collaboration grant \emph{Moduli of
Varieties}. U.W. was supported by NSF grant DMS-2100288.}



\begin{abstract}
We prove that hypersurfaces defined by irreducible square-free polynomials have rational singularities. As an easy consequence, we deduce that certain (possibly non-square-free) polynomials associated to pairs of square-free polynomials define hypersurfaces with rational singularities. This extends results on certain classes of polynomials associated to matroids and Feynman diagrams in \cite{BW}. 
\end{abstract}

\maketitle

\section{Introduction}

In \cite{BW}, two of the authors of the present paper introduced and studied several classes of polynomials that can be attached 
to matroids (for example, the \emph{matroid support polynomials} and the more general class of \emph{matroidal polynomials}). 
One of the main results in \emph{loc.\ cit}. says that if the matroid is connected and of positive rank (which implies that the polynomial is irreducible), then the corresponding hypersurface has rational singularities. For matroidal polynomials (see \cite[Corollary~4.29]{BW}),
this was done using an intricate analysis of the jet schemes of the hypersurface and the characterization of rational hypersurface singularities via jet schemes 
in \cite{Mustata}.

Our main result says that, in fact, all irreducible square-free polynomials define hypersurfaces with rational singularities. We work over
an algebraically closed field $\k$ of characteristic $0$. Recall that a nonzero polynomial $f\in \k[x_1,\ldots,x_n]$ is \emph{square-free} if 
every monomial that appears in $f$ has degree $\leq 1$ with respect to each variable. 

\begin{thm}\label{thm_main}
If $Z\subset \A^n$ is the hypersurface defined by an irreducible square-free polynomial $f\in \k[x_1,\ldots,x_n]$,
then $Z$ has rational singularities.
\end{thm}

We note that the assumption that $f$ is irreducible is natural: having rational singularities implies that the hypersurface is normal
and, in this case, it is easy to see that since $f$ is square-free, it has to be irreducible (see Remark~\ref{rem_normal} below).

In characteristic $p>0$, an analogue of rational singularities is provided by \emph{$F$-rational singularities}, a concept defined using the Frobenius 
homomorphism (see \cite{Smith1}). It is a natural question whether the following positive characteristic version of the above result holds\footnote{Shortly after the first version of this article was made public, a positive answer to this question was provided in \cite{Conca}.}:

\begin{question}\label{question1}
If $\k$ is an algebraically closed field of characteristic $p>0$ and $Z\subset\A^n$ is the hypersurface defined by an irreducible square-free polynomial $f\in \k[x_1,\ldots,x_n]$, does $Z$ have $F$-rational singularities?
\end{question}

We note that in \cite[Theorem~3.4]{BW}, one shows that the above question has a positive answer for the matroid support polynomials associated to connected matroids
of positive rank. 
We also note that a positive answer to Question~\ref{question1} would give another proof for Theorem~\ref{thm_main}, since by results of \cite{Smith2},
\cite{Hara}, and \cite{MS}, it is known that a variety $Z$ in characteristic $0$ has rational singularities if and only if its reduction mod $p$
has $F$-rational singularities for $p\gg 0$.

In \cite{BW}, one also attaches certain non-square-free polynomials (\emph{Feynman integrands} and \emph{Feynman diagram polynomials}) 
to certain combinatorial data. Under suitable assumptions, it is shown in 
\cite[Corollary~6.39 and Theorem~6.44]{BW} that the corresponding hypersurfaces
have rational singularities. 
The following result shows that this is a more general phenomenon (see the discussion before Proposition~\ref{last_prop}).
As before, we assume that $\k$ is algebraically closed, of characteristic $0$.

\begin{thm}\label{thm2}
Let $g, h\in \k[x_1,\ldots,x_n]$ be square-free polynomials, with ${\rm deg}(h)=1+{\rm deg}(g)$, and let
$L=a_1x_1+\ldots+a_nx_n+1$, for some $a_1,\ldots,a_n\in \k$. If $g$ is irreducible and does not divide $h$, then 
the hypersurface $Z$ defined in $\A^n$ by $gL+h$ has rational singularities.
\end{thm}

The proof of Theorem~\ref{thm_main} makes use of the minimal log discrepancy ${\rm mld}_{\eta_W}(\A^n,Z)$ at the generic point $\eta_W$
of an irreducible closed subset $W\subseteq\A^n$
(for the definition and basic properties of this invariant, see Section~\ref{section_review} below). By a standard result,
$Z$ has rational singularities if and only if ${\rm mld}_{\eta_W}(\A^n,Z)\geq 1$ for all irreducible closed subsets $W$ of $\A^n$
of codimension $r\geq 2$. The key is to show the following two properties, where we denote by ${\rm mult}_W(Z)$ the multiplicity of $Z$ along $W$:
\begin{enumerate}
\item[i)] If $f$ is square-free, then ${\rm mld}_P(\A^n,Z)\geq n-{\rm mult}_P(Z)$ for every $P\in \A^n$.
\item[ii)] If, in addition, $f$ is also irreducible, then ${\rm mult}_W(Z)\leq r-1$.    
\end{enumerate}
The fact that ${\rm mld}_{\eta_W}(\A^n,Z)\geq 1$ follows easily from i) and ii) and the fact that for $P\in W$ general, 
we have ${\rm mld}_{\eta_W}(\A^n,Z)={\rm mld}_P(\A^n,Z)-\dim(W)$ . This gives the conclusion of Theorem~\ref{thm_main}. The assertion in Theorem~\ref{thm2} follows easily from that in Theorem~\ref{thm_main}
by homogenization.

\section{Review of minimal log discrepancies}\label{section_review}

For the proof of Theorem~\ref{thm_main}, it will be convenient to use the notion of minimal log discrepancy.
We review in this section its definition, following \cite{Ambro}, and its connection to certain classes of singularities of pairs, following
\cite{Kollar}.

We work over a fixed algebraically closed field $\k$ of characteristic $0$. Since this will be enough for our purpose,
we only consider the case when the ambient variety $X$ is smooth and connected. Let $n=\dim(X)$. If $\pi\colon Y\to X$ is a birational
morphism, with $Y$ smooth, and $E$ is a prime divisor on $Y$, we
have a corresponding valuation ${\rm ord}_E$ of the
function field $\k(Y)=\k(X)$. We will refer to such $E$ as a \emph{divisor over} $X$. The \emph{center} of $E$ on $X$ is $c_X(E):=\overline{\pi(E)}$.
Note that if $Z$ is a hypersurface in $X$ defined by $f\in\cO_X(X)$, then ${\rm ord}_E(f)$ is the coefficient of $E$ in $\pi^*(Z)$.
Because of this, for every divisor $Z$ on $X$ (possibly non-effective), we denote this coefficient by ${\rm ord}_E(Z)$.

Given $\pi$ as above, the \emph{relative canonical divisor} $K_{Y/X}$ is the effective divisor on $Y$ defined by the determinant of
the morphism of rank $n$ locally free sheaves $\pi^*(\Omega_X)\to\Omega_Y$. If $E$ is a prime divisor on $Y$, we write 
$A_X({\rm ord}_E)$ for the coefficient of $E$ in $K_{Y/X}$ plus $1$. It is easy to check that $A_X({\rm ord}_E)$ only depends on the valuation ${\rm ord}_E$ and not on the
choice of $(Y,E)$. Given a divisor $Z$ on $X$ (not necessarily effective), 
the \emph{log discrepancy} of $(X,Z)$ with respect to $E$, denoted by $a_E(X,Z)$, is given by $A_X({\rm ord}_E)-{\rm ord}_E(Z)$. 

Given a closed subset $W$ of $X$, we define
$${\rm mld}_W(X,Z):=\inf\big\{a_E(X,Z)\mid E\,\,\text{is a divisor over}\,\,X\,\,\text{with}\,\,c_X(E)\subseteq W\big\}.$$
Moreover, if $W$ is an irreducible proper closed subset of $X$, we also consider
$${\rm mld}_{\eta_W}(X,Z):=\inf\big\{a_E(X,Z)\mid E\,\,\text{is a divisor over}\,\,X\,\,\text{with}\,\,c_X(E)=W\big\}.$$
Note that these are both integers or $-\infty$. If $W=\{P\}$, for some 
(closed) point $P\in X$, then we write ${\rm mld}_P(X,Z)$
for ${\rm mld}_W(X,Z)={\rm mld}_{\eta_W}(X,Z)$.
These notions can be extended to the case when $X$ is normal and ${\mathbb Q}$-Gorenstein and $Z$ is a divisor with rational coefficients,
but we will not need this level of generality. 

\begin{rmk}
We note that if $Z'$ is an effective divisor on $X$, then ${\rm ord}_E(Z')\geq 0$ for all divisors $E$ over $X$, and thus 
${\rm mld}_W(X,Z+Z')\leq {\rm mld}_W(X,Z)$ for all $W$. A similar inequality holds for ${\rm mld}_{\eta_W}$.
\end{rmk}

We note that if $Y\overset{h}\to V\overset{g}\to X$ are proper, birational morphisms between smooth varieties, then
$$K_{Y/X}=K_{Y/V}+h^*(K_{V/X}).$$
We thus have 
$$A_X({\rm ord}_E)=A_V({\rm ord}_E)+{\rm ord}_E(K_{V/X})$$
for every divisor $E$ over $X$. By definition of minimal log discrepancies, this implies that for every
closed subset $W\subseteq X$ and every divisor $Z$ on $X$, we have 
\begin{equation}\label{formula_mld}
{\rm mld}_W(X,Z)={\rm mld}_{g^{-1}(W)}\big(V,g^*(Z)-K_{V/X}\big).
\end{equation}

Several classes of singularities of pairs can be described in terms of minimal log discrepancies. 
Suppose now that $Z$ is an effective divisor on $X$.
The pair $(X,Z)$ is 
\emph{log canonical} if and only if ${\rm mld}_X(X,Z)\geq 0$, see \cite[Section~3]{Kollar}. The following well-known fact will be important for us:

\begin{thm}\label{thm_ingredient}
If $Z$ is a hypersurface in the smooth variety $X$, then $Z$ has rational singularities if and only if
${\rm mld}_W(X,Z)\geq 1$ for every irreducible closed subset $W$ of $X$ with ${\rm codim}_X(W)\geq 2$.    
\end{thm}

\begin{proof}
The fact that ${\rm mld}_W(X,Z)\geq 1$ for every irreducible closed subset $W$ of $X$ with ${\rm codim}_X(W)\geq 2$
is a reformulation of the fact that the pair $(X,Z)$ has canonical singularities (see \cite[Section 3]{Kollar}). By a version of Inversion of Adjunction
due to Stevens (see \cite[Theorem~7.9]{Kollar}), this is equivalent to $Z$ having canonical singularities. Since $Z$ is Gorenstein, this in turn is equivalent to $Z$ having rational singularities by a result of Elkik (see \cite[Corollary~11.13]{Kollar}).
\end{proof}

\section{Proofs of the main results}

In this section, too, we assume that we work over an algebraically closed field $\k$, of characteristic $0$.
We begin with a few easy statements.

\begin{rmk}\label{rem_normal}
If $f\in \k[x_1,\ldots,x_n]$ is a square-free polynomial that defines a nonempty normal hypersurface $Z$, then $f$ is irreducible. 
Indeed, since $Z$ is normal, it is reduced and its irreducible components are disjoint. Since $f$ is square-free and non-invertible, if $f$ is reducible, then we can write
$f=gh$, where $g$ and $h$ are non-invertible, involving disjoint sets of variables (consider the degrees with respect to each variable).
In this case, the hypersurfaces corresponding to $g$ and $h$ would meet nontrivially. 
\end{rmk}

\begin{lem}\label{lem_change_of_variable}
If $f\in \k[x_1,\ldots,x_n]$ is a square-free polynomial and we put $y_i=x_i+a_i$, with $a_i\in k$ for $1\leq i\leq n$, then
$f$ is square-free also as an element of $\k[y_1,\ldots,y_n]$.
\end{lem}

\begin{proof}
    The assertion follows directly from the definition.
\end{proof}

\begin{lem}\label{lem_log_canonical}
If $Z$ is a hypersurface in $\A^n$ defined by a square-free polynomial $f\in \k[x_1,\ldots,x_n]$, then the pair $(\A^n,Z)$ is log canonical.
\end{lem}

\begin{proof}
We argue by induction on $n\geq 1$. If $n=1$, then $Z$ is either empty or a reduced point, and the assertion is clear. 
Suppose now that $n\geq 2$. It is enough to show that the pair $(\A^n,Z)$ has log canonical singularities around every point $P=(a_1,\ldots,a_n)\in\A^n$.
After replacing $x_i$ by $y_i=x_i-a_i$ for $1\leq i\leq n$, and using Lemma~\ref{lem_change_of_variable}, we see that we may assume that $P$ is the origin.
Given $i$, with $1\leq i\leq n$, if $H_i$ is the hyperplane defined by $(x_i=0)$ and $x_i$ does not divide $f$, then it is clear that $f\vert_{H_i}$ is square-free, hence the
pair $(H_i,Z\vert_{H_i})$ is log canonical by the inductive hypothesis. Inversion of Adjunction (see \cite[Theorem~7.5]{Kollar}) implies 
that the pair $(\A^n,Z+H_i)$ is log canonical in a neighborhood of $H_i$, and thus $(\A^n,Z)$ is log canonical in a neighborhood of $H_i$
(and, thus, in a neighborhood of $P$). On the other hand, if $x_i$ divides $f$ for $1\leq i\leq n$, then $f$ is a scalar multiple of $x_1\cdots x_n$.
In this case $Z=\sum_{i=1}^nH_i$ is a simple normal crossing divisor and the pair $(\A^n,Z)$ is clearly log canonical. 
\end{proof}

We now can show that hypersurfaces defined by irreducible square-free polynomials have rational singularities. 

\begin{proof}[Proof of Theorem~\ref{thm_main}]
The case $n\le 1$ is trivial, so we shall assume $n\geq 2$. 
For a (closed) point $P\in \A^n$, we denote by ${\rm mult}_P(Z)$ the multiplicity of $Z$ at $P$ (that is, the largest $q$ such that 
$f\in {\mathfrak m}_P^q$, where ${\mathfrak m}_P$ is the ideal defining $P$).
If $W$ is a proper irreducible closed subset of $\A_k^n$,
we denote by ${\rm mult}_W(Z)$ the multiplicity of $Z$ at a general point of $W$. Since we are in characteristic $0$, this can be described
as the largest $q$ such that $\tfrac{\partial^{\alpha} f}{\partial x^{\alpha}}$ vanishes on $W$ for all $\alpha=(\alpha_1,\ldots,\alpha_n)$
with $\sum_i\alpha_i\leq q-1$.
We will show that the following hold:
\begin{claim}\label{claim1}
If $f$ is square-free, then for every point $P\in \A^n$, we have ${\rm mld}_P(\A^n,Z)\geq n-{\rm mult}_P(Z)$.    
\end{claim}
\begin{claim}\label{claim3}
If $f$ is irreducible and square-free, then $Z$ is normal.
\end{claim}
\begin{claim}\label{claim2}
If $f$ is irreducible and square-free, and $W$ is a closed, irreducible subset of $\A^n$ of codimension $r\geq 2$, then ${\rm mult}_W(Z)\leq r-1$.
\end{claim}

Let's prove first Claim~\ref{claim1}. As in the proof of Lemma~\ref{lem_log_canonical}, using a suitable change of variables and Lemma~\ref{lem_change_of_variable},
we may assume that $P$ is the origin. Let $g\colon V\to \A^n$ be the blow-up of $\A^n$ at $P$. We denote by $E$ the exceptional divisor and by $\widetilde{Z}$ the strict
transform of $Z$, so we have $g^*(Z)=\widetilde{Z}+mE$, where $m={\rm mult}_P(Z)$. Since $K_{V/\A^n}=(n-1)E$, it follows from (\ref{formula_mld}) that
\begin{equation}\label{eq_blow_up}
{\rm mld}_P(\A^n,Z)={\rm mld}_E\big(V,g^*(Z)-K_{V/\A^n}\big)={\rm mld}_E\big(V,\widetilde{Z}-(n-1-m)E\big).
\end{equation}

Let us write $f=\sum_{\ell\geq m}f_{\ell}$, where each $f_{\ell}$ is homogeneous of degree $\ell$. 
We note that $V=V_1\cup\ldots\cup V_n$, where $V_i\simeq\A^n$ is the chart on $V$ with coordinates $y_1,\ldots,y_n$ such that $x_i=y_i$ and $x_j=y_iy_j$ for all $j\neq i$. Note that $E\cap V_i$ is the hyperplane defined by $(y_i=0)$. Since $\widetilde{Z}\vert_E\cap V_i$
is defined in $E\cap V_i$ by $f_m(y_1,\ldots,1,\ldots,y_n)$, where $1$ appears in the $i$th spot, and this is a square-free polynomial, we deduce via Lemma~\ref{lem_log_canonical}
that the pair $(E,\widetilde{Z}\vert_E)$ is log canonical. By Inversion of Adjunction (see \cite[Theorem~7.5]{Kollar}), we conclude that the pair
$(V,\widetilde{Z}+E)$ is log canonical in a neighborhood of $E$. This implies that if $F$ is a divisor over $V$, with $c_V(F)\subseteq E$, then 
$$a_F\big(V,\widetilde{Z}-(n-1-m)E\big)=a_F(V,\widetilde{Z}+E)+(n-m)\cdot {\rm ord}_F(E)\geq n-m,$$ using the fact that
$a_F(V,\widetilde{Z}+E)\geq 0$, since $(V,\widetilde{Z}+E)$ is log canonical in a neighborhood of $E$, and 
${\rm ord}_F(E)\geq 1$, since $c_V(F)\subseteq E$. Using (\ref{eq_blow_up}), we conclude that
${\rm mld}_P(\A^n,Z)\geq n-m$, completing the proof of Claim~\ref{claim1}.

Let us prove Claim~\ref{claim3}.  Since $Z$ is a hypersurface, it is Cohen--Macaulay, hence by Serre's criterion,
in order to prove that $Z$ is normal, it is enough to show that the codimension in $Z$ of the singular locus $Z_{\rm sing}$ of $Z$ is $\geq 2$. 
Arguing by contradiction, suppose that $Z_0$ is an irreducible component of $Z_{\rm sing}$ with $\dim(Z_0)\geq n-2$. We may and will assume that all
variables $x_i$ appear in $f$. Since $f$ is square-free, for every $i$, we may write $f=g_ix_i+h_i$, with $g_i,h_i\in \k[x_1,\ldots,x_{i-1},x_{i+1},\ldots,x_n]$.
By assumption, we have $g_i\neq 0$, and we may also assume $h_i\neq 0$ (otherwise, since $f$ is irreducible, it follows that $g_i\in \k$ and we are done).
Since $Z_0$ is contained in the zero-locus of $\tfrac{\partial f}{\partial x_i}$, it follows that it is contained in the zero-locus of $(g_i,h_i)$.
The irreducibility of $f$ implies that the hypersurfaces defined by $g_i$ and $h_i$ have no common irreducible components,
hence $Z_0$ is an irreducible component of $V(g_i,h_i)$. We thus see that 
if $\pi_i\colon\A_k^n\to\A_k^{n-1}$ is the projection that forgets the $i$th component, then $Z_0$ is the inverse image via $\pi_i$
of a closed subset of $\A_k^{n-1}$. Since this holds for all $i$, it follows that $Z_0=\A_k^n$, a contradiction. This completes the proof of the fact that $Z$ is normal.

We next prove Claim~\ref{claim2} by induction on $d=\dim(W)$. If $d=0$, then $r=n$ and $W=\{P\}$ is a point.
Using again a suitable change of variables and Lemma~\ref{lem_change_of_variable}, we may assume that $P$ is origin.
In this case, since f is square-free,
it is clear that ${\rm mult}_P(f)\leq n$. Moreover, this is an equality if and only if $f$ is a scalar multiple of $x_1\cdots x_n$.
However, this can't happen when $n\geq 2$, since $f$ is assumed to be irreducible (here is where we use the hypothesis $r\geq 2$).

 Suppose now that $d\geq 1$. After relabeling the coordinates, we may and will assume that the projection of $W$ to the first component is not constant. We may also assume that $W\subseteq Z$, since otherwise the inequality to be proved is clear. 
By hypothesis, if $H_t$ is the hyperplane in $\A^n_k$
given by $(x_1-t=0)$, for $t\in \k$ general, then $W\cap H_t$ is nonempty, of pure dimension $d-1$. It is clear that, with respect to the restriction of the variables
$x_2,\ldots,x_n$, the polynomial $f\vert_{H_t}$ is square-free; when $t$ is generic, then it is also irreducible. Indeed, it follows from the Kleiman--Bertini theorem
(recall that we are in characteristic $0$) that the singular locus of $Z\cap H_t$ is contained in $Z_{\rm sing}\cap H_t$. For general $t$, the latter set
has dimension $\leq \dim(Z_{\rm sing})-1\leq n-4$, where the last inequality follows from Claim~\ref{claim3}. By Serre's criterion, this implies that
$Z\cap H_t$ is normal, and thus $f\vert_{H_t}$ is irreducible by Remark~\ref{rem_normal}. 

Let $W_t$ be an irreducible component of $W\cap H_t$. Since $\dim(W_t)=d-1$, we may apply the inductive hypothesis for $f\vert_{H_t}$ and $W_t$.
Note that ${\rm codim}_{H_t}(W_t)=r$, and thus we conclude that
$${\rm mult}_W(Z)\leq {\rm mult}_{W_t}(Z\cap H_t)\leq r-1$$
(while we do not need this, we note that the first inequality above is, in fact, an equality, since the hyperplane $H_t$ is general).
This completes the proof of Claim~\ref{claim2}.

We now combine the assertions in Claims~\ref{claim1} and \ref{claim2} to show that $Z$ has rational singularities. By Theorem~\ref{thm_ingredient},
it is enough to show that for every irreducible closed subset $W$ of $\A^n$ with ${\rm codim}_{\A^n}(W)=r\geq 2$, we have 
${\rm mld}_W(\A^n,Z)\geq 1$. Since
$${\rm mld}_W(\A^n,Z)=\inf_{W'\subseteq W}\big\{{\rm mld}_{\eta_{W'}}(\A^n,Z)\big\},$$
it is enough to show that for all $W$ as above, we have
${\rm mld}_{\eta_W}(\A^n,Z)\geq 1$. 

Let $P\in W$ be a general point, so that ${\rm mult}_W(Z)={\rm mult}_P(Z)$. 
Since the pair $(X,Z)$ is log canonical by Lemma~\ref{lem_log_canonical}
and since $P\in W$ is a general point, we have 
\begin{equation}\label{eq1_pf_thm_main}
{\rm mld}_{\eta_W}(\A^n,Z)={\rm mld}_P(\A^n,Z)-\dim(W),
\end{equation}
see \cite[Proposition~2.3]{Ambro}.
By Claim~\ref{claim1}, we have
\begin{equation}\label{eq2_pf_thm_main}
{\rm mld}_P(\A^n,Z)\geq n-{\rm mult}_P(Z)=n-{\rm mult}_W(Z)\geq n-r+1,
\end{equation}
where the last inequality follows from Claim~\ref{claim2}. By combining (\ref{eq1_pf_thm_main}) and (\ref{eq2_pf_thm_main}),
we obtain ${\rm mld}_{\eta_W}(\A^n,Z)\geq 1$, completing the proof of the theorem.
\end{proof}

\begin{rmk}
The result in Theorem~\ref{thm_main} was recently extended by Supravat Sarkar \cite{Sarkar}, who showed that irreducible divisors in ${\mathbb P}^{n_1}\times\ldots\times {\mathbb P}^{n_r}$
of type $(1,\ldots,1)$ have rational singularities. His proof relies on techniques from the Minimal Model Program.
\end{rmk}

\begin{rmk}
After the first version of this article was made public, Matt Larson pointed out to us that, in fact, one can use deduce the assertion in Theorem~\ref{thm_main} 
(and also the above-mentioned result of Sarkar)
from 
a more general result of Michel Brion (see \cite[Theorem~5]{Brion1} and also
\cite[Remark~3]{Brion2}). Brion's concerns multiplicity-free subvarieties in homogeneous spaces. In our setting, it applies to the closure $\Gamma$ in $({\mathbb P}^1)^n$ of the hypersurface defined by $f$. Since $f$ is square-free, it follows that
$\Gamma$ is multiplicity-free, in the sense that its cohomology class written in terms of the usual basis for the cohomology of $({\mathbb P}^1)^n$ has only $0$ or $1$ coefficients. Brion's result then implies that since $\Gamma$ is irreducible and reduced, it has rational singularities. 
\end{rmk}

We next turn to the second result stated in the Introduction.

\begin{proof}[Proof of Theorem~\ref{thm2}]
We put $d={\rm deg}(g)$. Let $L'=a_1x_1+\ldots+a_nx_n+x_0$ and $f'=gL'+h\in \k[x_0,x_1,\ldots,x_n]$. 
We denote by $Z'$ the hypersurface defined by $f'$ in $\A^{n+1}$.
By assumption, $f'$ is homogeneous, of degree $d+1$. If $Z'$ has rational singularities, then so does $Z$
(indeed, the open subset $U$ of $\A^{n+1}$ defined by $(x_0\neq 0)$ is isomorphic to $\A^n\times \big(\A^1\smallsetminus\{0\}\big)$,
such that $Z'\cap U$ coincides with the pull-back of $Z$ via the first projection). 

We consider the change of coordinates $y_i=x_i$ for $1\leq i\leq n$
and $y_0=L'(x)$. Since $g$ and $h$ are square-free polynomials, it is clear that $f'$ is square-free with respect 
to the coordinates $y_0,\ldots,y_n$. Moreover, $f'$ is irreducible. Indeed, suppose that we can write $f'=PQ$, with 
$P,Q\in \k[y_0,\ldots,y_n]$ of positive total degree. Since ${\rm deg}_{y_0}(f')=1$, we may assume that
${\rm deg}_{y_0}(P)=0$ and ${\rm deg}_{y_0}(Q)=1$. If we write $Q=Q_1y_0+Q_2$, with $Q_1,Q_2\in \k[x_1,\ldots,x_n]$,
it follows that $g=PQ_1$ and $h=PQ_2$. Since $g$ is irreducible, it follows that $Q_1\in \k$, and thus $g$ divides $h$, a contradiction.

Since $f'$ is irreducible and square-free, we may apply Theorem~\ref{thm_main} to conclude that $Z'$ has rational singularities.
As we have seen, this implies the assertion in the theorem.
\end{proof}

\section{Comparison with the results in \cite{BW}}

We end by making the connection with the polynomial invariants associated to matroids and Feynman diagrams in \cite{BW}. For the relevant definitions related to matroid theory,
we refer to \cite[Section 2]{BW}. 

If $M$ is a matroid on the set $E=\{1,\ldots,n\}$, then one defines in \cite[Definition~2.11]{BW} a \emph{matroid support polynomial} of $M$
as a polynomial $\xi_M\in \k[x_1,\ldots,x_n]$
of the form $\xi_M=\sum_Bc_Bx^B$, where $B$ is running over the bases of $M$, $c_B\in \k$ is nonzero for all $B$, and $x^B=\prod_{i\in B}x_i$.
It is clear that $\xi_M$ is a square-free polynomial. Moreover, if $M$ is a connected matroid, of positive rank (so $\xi_M$ is not invertible),
then $\xi_M$ is irreducible. Indeed,
if $\xi_M=PQ$, for non-invertible polynomials $P$ and $Q$, then $P$ and $Q$ involve disjoint sets of variables. We thus have a decomposition
$E=E'\sqcup E''$, with $E'$ and $E''$ nonempty, and corresponding matroids $M'$ on $E'$ and $M''$ on $E''$ such that the bases of $B$ consist of 
subsets of the form $B'\cup B''$, where $B'$ is a basis of $M'$ and $B''$ is a basis of $M''$. This contradicts the fact that $E$ is connected.
We can thus apply Theorem~\ref{thm_main} to conclude that the hypersurface defined by a matroid support polynomial of $M$ has
rational singularities. Note that \cite[Theorem~3.4]{BW} gives the (stronger) result saying that in positive characteristic, the hypersurface defined
by such $\xi_M$ has $F$-rational (or equivalently, strongly $F$-regular) singularities. 

More generally, one defines in \cite[Definition~2.15]{BW} the notion of \emph{matroidal polynomial} associated to a matroid $M$ and certain extra data 
(called \emph{singleton data}). This is again a square-free polynomial and if $M$ is connected, of positive rank, then this polynomial is irreducible
by \cite[Corollary~2.19]{BW}. We can thus apply Theorem~\ref{thm_main} to deduce that the corresponding hypersurface has rational singularities, recovering
\cite[Corollary~4.29]{BW}.

In \cite{BW}, the authors also consider certain polynomials coming from Feynman diagrams and Feynman integrals, which are interesting to mathematical physicists due to their importance in Quantum Field Theory. A Feynman diagram is a graph $(V,E)$ decorated with various mass data $\mathbf{m}$ and momenta data $p$; in Lee--Pomeransky form, the Feynman integral is the Mellin transform of the \emph{Feynman diagram polynomial} $\mathscr{G} = \mathscr{U}(1 + \Delta_{\mathbf{m}}^E) + F_0^W$ (see \cite[Definition~6.41]{BW} for a precise definition and \cite[Section~6.4]{BW} for a quick survey of the physics set-up). Feynman diagram polynomials fit into the set-up of Theorem \ref{thm2}: $\mathscr{U}$ is a matroid support polynomial for the cographic matroid $M_G^\perp$ associated to the underlying (undecorated) graph and plays the role of $g$; $1 + \Delta_{\mathbf{m}}^E$ is linear and plays the role of $L$; $F_0^W$ is squarefree with $\deg(F_0^W) = \deg(\mathscr{U}) + 1$ and plays the role of $h$. Under mild genericity conditions on mass and momenta (see \cite[Proposition~6.43]{BW}), the Feynman diagram polynomial is an instance of a so-called \emph{Feynman integrand} ${\rm Feyn}(\zeta_N, \Delta_{\mathbf{m}}, \xi_M)$, see \cite[Definition~6.1]{BW}. Provided the underlying matroid is connected and of positive rank\footnote{The assumption in \cite{BW} is that the rank of the matroid is $\geq 2$, in order to avoid the possibility that the hypersurface is smooth. However, in this note, we use the convention that a smooth variety has rational singularities.}, it is proven in \cite[Corollary~6.39]{BW} that all Feynman integrands have rational singularities and so, under the aforementioned genericity assumptions, all Feynman diagram polynomials have rational singularities \cite[Theorem~6.44]{BW}. Theorem \ref{thm2} allows the extension of the result on Feynman diagram polynomials, by removing the genericity assumptions on mass and momenta data, as follows:

\begin{prop}\label{last_prop}
Let $G=G(V,E,\mathbf{m}, p)$ be a Feynman diagram and $\mathscr{G}$ the corresponding Feynman diagram polynomial. If the underlying graphic matroid $M_G$ of the graph $(V,E)$ is connected, of positive rank, then the hypersurface defined by $\mathscr{G}$ has rational singularities. 
\end{prop}

\begin{proof}
As described above, we can put $\mathscr{G}$ in the setting of Theorem \ref{thm2}. Indeed, we have $\mathscr{G} = gL + h$, where $g$ is a matroid support polynomial for the cographic matroid $M_G^\perp$, $g$ and $h$ are square-free, and $\deg(h) = \deg(g) + 1$. By Theorem \ref{thm2}, it suffices to verify that $g$ is irreducible and $g$ does not divide $h$. Note that $g$ is not invertible since we assume the matroid has positive rank. On the other hand, the connectedness of the graphic matroid $M_G$ is equivalent to the connectedness of the cographic matroid $M_G^\perp$, and we have already seen that the latter's connectedness implies that $g$ is irreducible. Moreover, since $M_G^\perp$ is connected and in particular loopless, all variables $x_i$ appear in $g$. This implies $g$ does not divide $h$. For if $h = gP$, then $g$ and $P$ use disjoint variables as $h$ and $g$ are square-free. However, since all variables appear in $g$, it follows that
$P$ is a constant, contradicting $\deg(h) = \deg(g) + 1$.
\end{proof}

\begin{rmk}
An entirely similar argument shows that Theorem \ref{thm2} recovers the rational singularity result \cite[Corollary~6.39]{BW} for Feynman integrands.
\end{rmk}

\begin{rmk}
One can ask (and indeed, we were asked this question) whether hypersurfaces defined by irreducible Lorentzian polynomials in the sense of
\cite{HuhBraden} have rational singularities. This is a natural question: a Lorentzian polynomial has $M$-convex monomial support and the intersection of square-free polynomials and polynomials with $M$-convex monomial support are exactly the matroid support polynomials. However, the answer is negative: there are Lorentzian polynomials that remain irreducible over ${\mathbb C}$ whose corresponding complex hypersurface $Z$ in ${\mathbb A}^n$ doesn't have rational singularities. In fact, we might not even have
$({\mathbb A}^n,Z)$ log canonical. 
This is due to the fact that Lorentzian polynomials (which are always homogeneous) have no intrinsic restriction on their degrees, whereas the condition $({\mathbb A}^n, Z)$ log canonical requires the degree of the polynomial to be $\leq n$ (this follows by 
considering the exceptional divisor on the blow-up of ${\mathbb A}^n$ at $0$). 
For example, if $Z$ is the complex hypersurface in ${\mathbb A}^3$ defined by the Lorentzian polynomial
\begin{equation*}
    \sum_{\{(a,b,c) \in \mathbb{Z}_{\geq 0}^3 \,|\, a + b + c = 4\}} \frac{1}{a! \cdot b! \cdot c!} x^a y^b z^c\in {\mathbb R}[x,y,z]
\end{equation*}
(see \cite[Theorem~3.10]{HuhBraden}), which remains irreducible over ${\mathbb C}$, 
then $({\mathbb A}^n,Z)$ is not log canonical. 
\end{rmk}

\noindent {\bf Acknowledgment}. We are grateful to Matt Larson for pointing out to us the reference \cite{Brion1}.

\end{document}